\documentclass[a4paper,10pt,oneside]{amsart}
\usepackage[DIV=14,BCOR=2mm,headinclude=false,footinclude=false]{typearea}
\usepackage{tikz}
\usepackage[latin1]{inputenc}
\usepackage{amsfonts}
\usepackage{amsmath}
\usepackage{amsthm}
\usepackage{amssymb}
\usepackage{mathtools}
\usepackage{amsbsy, hyperref}
\usepackage{mathrsfs}
\usepackage{dsfont}
\usepackage{tcolorbox}
\usepackage{comment}
\usepackage{bbm}
\usepackage{enumerate}
\usepackage[numbers]{natbib}
\usepackage{graphicx,color}
\usepackage{yfonts}
\usepackage[labelfont=rm,format=plain,indention=0cm,singlelinecheck=off,justification=raggedright,skip=2pt]{caption}
\usepackage[fit]{truncate}
\usepackage{placeins} 
\usepackage{flafter}  

\usetikzlibrary{calc,matrix,backgrounds}

\pgfdeclarelayer{overlay}
\pgfsetlayers{overlay,background,main}

\tikzset{circle/.style = {rounded corners,line width=1bp,color=#1}}%

\makeatletter
\renewcommand{\tocsection}[3]{%
  \indentlabel{\@ifnotempty{#2}{\bfseries\ignorespaces#1 #2\quad}}\bfseries#3}
\renewcommand{\tocsubsection}[3]{%
  \indentlabel{\@ifnotempty{#2}{\ignorespaces#1 #2\quad}}#3}

\newcommand\@dotsep{4.5}
\def\@tocline#1#2#3#4#5#6#7{\relax
  \ifnum #1>\c@tocdepth 
  \else
    \par \addpenalty\@secpenalty\addvspace{#2}%
    \begingroup \hyphenpenalty\@M
    \@ifempty{#4}{%
      \@tempdima\csname r@tocindent\number#1\endcsname\relax
    }{%
      \@tempdima#4\relax
    }%
    \parindent\z@ \leftskip#3\relax \advance\leftskip\@tempdima\relax
    \rightskip\@pnumwidth plus1em \parfillskip-\@pnumwidth
    #5\leavevmode\hskip-\@tempdima{#6}\nobreak
    \leaders\hbox{$\m@th\mkern \@dotsep mu\hbox{.}\mkern \@dotsep mu$}\hfill
    \nobreak
    \hbox to\@pnumwidth{\@tocpagenum{\ifnum#1=1\bfseries\fi#7}}\par
    \nobreak
    \endgroup
  \fi}
\AtBeginDocument{%
\expandafter\renewcommand\csname r@tocindent0\endcsname{0pt}
}
\def\l@subsection{\@tocline{2}{0pt}{2.5pc}{5pc}{}}
\makeatother

\newcommand{\Addresses}{{
  \bigskip
  \footnotesize
  Camil Muscalu, \textsc{Department of Mathematics, Cornell University,
    Ithaca, New York 14853}\par\nopagebreak
 \textit{E-mail address}: \texttt{camil@math.cornell.edu}\\
 
  Itamar Oliveira, \textsc{CNRS - Universit\'e de Nantes, Laboratoire de Math\'ematiques Jean Leray, Nantes 44322, France}\par\nopagebreak
 \textit{E-mail address}: \texttt{oliveira.itamar.w@gmail.com}

}}

\theoremstyle{plain}
\newtheorem{theorem}{Theorem}[section]

\newtheorem{lemma}[theorem]{Lemma}

\newtheorem{proposition}[theorem]{Proposition}

\theoremstyle{definition}

\newtheorem{remark}[theorem]{Remark}

\newtheorem{thmx}{Problem}

\def\eqalign#1{\null\,\vcenter{\openup\jot\mathsurround\dimen12
  \ialign{\strut\hfil$\textstyle{##}$&$\textstyle{{}##}$\hfil
      \crcr#1\crcr}}\,}

\begin{document}

\begin{abstract} This note presents a new proof of the well-known Strichartz estimates for the Schr\"odinger equation in $2+1$ dimensions, building on ideas from our recent work \cite{MO}.
\end{abstract}

\title[A new proof of Strichartz estimates for the Schr\"odinger equation in $2+1$ dimensions]{A new proof of Strichartz estimates for the Schr\"odinger equation in $2+1$ dimensions}
\author{Camil Muscalu and Itamar Oliveira}
\thanks{C.M. was partially supported by a grant from the Simons Foundation. I.O. is supported by ERC project FAnFArE no. 637510.}
\date{}
\maketitle

\tableofcontents

\section{Introduction}

Let $S$ be a subset of $\mathbb{R}^{n}$ and $\mathrm{d}\mu$ a positive measure supported on $S$ and of temperate growth at infinity. In \cite{Str1}, Strichartz considered the following problems that we quote \textit{verbatim}:

\begin{thmx} For which values of $p$, $1\leq p<2$, is it true that $f\in L^{p}(\mathbb{R}^{n})$ implies $\widehat{f}$ has a well-defined restriction to $S$ in $L^{2}(\mathrm{d}\mu)$ with
$$\left(\int |\widehat{f}|^{2}\mathrm{d}\mu\right)^{\frac{1}{2}}\leq c_{p}\|f\|_{p}? $$

\end{thmx}

\begin{thmx} For which values of $q$, $2<q\leq\infty$, is it true that the tempered distribution $F\mathrm{d}\mu$ for each $F\in L^{2}(\mathrm{d}\mu)$ has Fourier transform in $L^{q}(\mathbb{R}^{n})$ with
$$\|\widehat{(F\mathrm{d}\mu)}\|_{q} \leq c_{q}\left(\int|F|^{2}\mathrm{d}\mu\right)^{\frac{1}{2}} ? $$
\end{thmx}

As it is pointed out in \cite{Str1}, a duality argument shows that these problems are equivalent if $\frac{1}{p}+\frac{1}{q}=1$. In the same paper, Strichartz gave a complete answer to the problems above when $S$ is a quadratic surface given by

$$S=\{x\in\mathbb{R}^{n}:R(x)=r\}, $$
where $R(x)$ is a polynomial of degree two with real coefficients, $r$ is a real constant and $\mathrm{d}\mu$ is a certain canonical measure on $S$ associated to $R$. 

In general, given a compact submanifold $S\subset\mathbb{R}^{n}$ and a function $f:\mathbb{R}^{d+1}\mapsto\mathbb{C}$, the \textit{Fourier restriction problem} asks for which pairs $(p,q)$ one has\footnote{We start with $f\in\mathcal{S}(\mathbb{R}^{d+1})$ so both $\widehat{f}$ and $\widehat{f}|_S$ are well-defined pointwise. If \eqref{restrictionstatement} holds for such $f$, we extend it to $L^{p}$ by density.}
\begin{equation}\label{restrictionstatement}
    \| \widehat{f}|_S\|_{L^{q}(S)}\lesssim \|f\|_{L^{p}(\mathbb{R}^{d+1})},
\end{equation}
where $\widehat{f}|_S$ is the restriction of the Fourier transform $\widehat{f}$ to $S$. This problem arises naturally in the study of certain Fourier summability methods and is known to be connected to questions in Geometric Measure Theory and in dispersive PDEs. We refer the reader to \cite{Tao-notes} for a detailed account of the problem. 

Here we specialize to the case $p=2$, $d=2$ and $S$ the compact piece of the paraboloid parametrized by $\Gamma(x)=(x,|x|^{2})\subset\mathbb{R}^{d+1}$, $x\in[0,1]^{d}$. As we briefly explain in the end of this introduction, this special case is closely related to the study of the Schr\"odinger equation. From \cite{Str1} and the necessary conditions for \eqref{restrictionstatement} to hold (see \cite{Tao-notes}), it follows that the restriction problem in this setting is equivalent up to the endpoint to showing the following:

\begin{theorem}[Tomas-Stein/Strichartz up to the endpoint]\label{mainthm} Let the Fourier extension operator $\mathcal{E}_{2}$ be defined on $C^{0}([0,1]^{2})$ by
$$\mathcal{E}_{2}(g)(x_{1},x_{2},t)=\int_{[0,1]^{2}}g(\xi_{1},\xi_{2})e^{-2\pi i(x_{1},x_{2})\cdot (\xi_{1},\xi_{2})}e^{-2\pi it(\xi_{1}^{2}+\xi_{2}^{2})}\mathrm{d}\xi_{1}\mathrm{d}\xi_{2}. $$
Then
\begin{equation}\label{bound1-jan1823}
\|\mathcal{E}_{2}g\|_{4+\varepsilon}\lesssim_{\varepsilon} \|g\|_{2}.
\end{equation}
\end{theorem}
The $d$-dimensional analogue of the operator above is
$$\mathcal{E}_{d}(g)(x,t)=\int_{[0,1]^{d}}g(\xi)e^{-2\pi ix\cdot \xi}e^{-2\pi it|\xi|^{2}}\mathrm{d}\xi. $$

Strichartz proved in \cite{Str1} that $\mathcal{E}_{d}$ maps $L^{2}$ to $L^{\frac{2(d+2)}{d}}$, so the only novelty of this manuscript is the argument in the case $d=2$ without the endpoint\footnote{We are not able to reach the endpoint case $\varepsilon=0$ due to the way we interpolate between some key information later in the proof. Roughly speaking, the condition $\varepsilon>0$ is necessary to make certain series converge.}. We use the framework of our earlier paper \cite{MO}, in which we propose a new approach to the linear and multilinear Fourier extension problems for the paraboloid.

\begin{remark} Given that Theorem \ref{mainthm} is known in much greater generality (with the endpoint, in arbitrary dimensions, for other underlying submanifolds, etc.), it is opportune to make a comparison between the classical approaches to it and ours. As explained in Chapter 11 of \cite{MS}, an estimate such as
\begin{equation*}
\|\mathcal{E}_{d}g\|_{q}\lesssim \|g\|_{2}
\end{equation*}
is equivalent to showing that
\begin{equation}\label{ineq1-jan2523}
\|f\ast\widehat{\mu}\|_{L^{q}(\mathbb{R}^{d+1})}\lesssim \|f\|_{L^{q^{\prime}}(\mathbb{R}^{d+1})},
\end{equation}
where $\mu$ is the measure in $\mathbb{R}^{d+1}$ defined by the integral
$$\int_{\mathbb{R}^{d+1}}h(\xi,\tau)\mathrm{d}\mu(\xi,\tau) = \int_{\mathbb{R}^{d}}h(\xi,|\xi|^{2})\mathrm{d}\xi.$$

The curvature of the paraboloid makes $\widehat{\mu}$ decay, which implies good integrability properties for the convolution operator $f\mapsto f\ast\widehat{\mu}$. This observation plays an important role in the proof of \eqref{ineq1-jan2523}.

Our strategy is different from the previous one; roughly speaking, instead of relying on a three-dimensional observation such as $\widehat{\mu}(\xi)=O(\langle\xi\rangle^{-\alpha})$ for a certain $\alpha>0$, we exploit some sharp one-dimensional bilinear estimates to prove Theorem \ref{mainthm}. The argument that we will present implies \eqref{bound1-jan1823} directly, without reducing it to an inequality such as \eqref{ineq1-jan2523}, in which the $L^{p}$ norms on both sides are taken over the same Euclidean space. On the other hand, the proof relies strongly on the tensor structure embedded in the paraboloid. More explicitly, we take advantage of the simple identity
$$e^{2\pi it|\xi|^{2}}=e^{2\pi it\xi_{1}^{2}}\cdot e^{2\pi it\xi_{2}^{2}}.$$

In short, the scheme of the proof is as follows:
\begin{enumerate}[(a)]
\item Bilinearize $\mathcal{E}_{2}$. In general, multilinearizing an oscillatory integral operator brings into play underlying geometric features of the problem. This is a common step in the study of the extension operator that goes back to Zygmund in \cite{Zyg} and Fefferman's thesis \cite{Fef1}. It also plays a central role in the work \cite{TVV1} of Tao, Vargas and Vega, and more recently in Bourgain and Guth's work \cite{BG}.
\item Decompose the previously mentioned bilinearization of $\mathcal{E}_{2}$ into certain pieces to gain access to key geometric considerations. This will be achieved via a classical Whitney decomposition.
\item Discretize the pieces obtained in the previous step in order to conveniently place them in the framework of our earlier paper \cite{MO}.
\item Study the ``Whitney pieces". Their study will be reduced to the lower-dimensional bilinear analogue of \eqref{bound1-jan1823}.
\end{enumerate}

The proof was motivated by our paper \cite{MO}, in which we show that the linear and multilinear extension conjectures are true (up to the endpoint\footnote{We obtain the endpoint in certain cases.}) if one of the functions involved is a full tensor. One of the main ideas was to also take advantage of known bounds in the $d=1$ case. In the ongoing process of understanding the extent to which the method of \cite{MO} applies, we learned that it could be used to prove Theorem \ref{mainthm} without any tensor hypothesis.

We hope that this new argument in the simple $d=2$ case will foster further interest in trying to answer some very natural questions: \textit{can one prove $n$-dimensional extension estimates from (linear and multilinear) lower dimensional ones?} Or at least \textit{what does (linear and multilinear) lower dimensional extension theory reveal about its higher dimensional analogues?}

\end{remark}

\begin{remark} There are a number of different proofs of Strichartz estimates in dimensions $d=1,2$, and also of their analogues for other hypersurfaces. In particular, we highlight the papers \cite{BBCH, F, G, HZ, PV}, which covers a wide range of tools and techniques in Analysis.
\end{remark}

We close this short introduction with the connection between \eqref{bound1-jan1823} and dispersive PDEs, as suggested by the title of the paper. Consider the linear Schr\"odinger equation\footnote{The constant $-\frac{1}{2\pi}$ is chosen for cosmetic reasons related to the normalization of the Fourier transform.}

\begin{equation}\label{Schr}
    \begin{dcases}
        iu_{t} - \frac{1}{2\pi}\Delta u = 0, \\
        u(x,0)=u_{0}(x),  \\
    \end{dcases}
\end{equation}
where $\widehat{u_{0}}\in C^{\infty}(B(0,1))$. The solution can be directly computed by solving an ODE after taking the Fourier transform in $x$, which gives

$$ u(x,t)=\int_{\mathbb{R}^{d}}e^{2\pi i(x\cdot\xi +t|\xi|^{2})}\widehat{u_{0}}(\xi)\mathrm{d}\xi = \mathcal{E}_{d}(\widehat{u_{0}})(-x,-t),$$
hence estimates for $\mathcal{E}_{d}$ imply estimates for $u$. Fourier extension theory, a scaling argument (to drop the compact support hypothesis) and the fact that $C^{\infty}(\mathbb{R}^{d})$ is dense in $L^{2}(\mathbb{R}^{d})$ allows us to conclude that
\begin{equation}\label{Strestimate}
    \|u\|_{L^{\frac{2(d+2)}{d}}(\mathbb{R}^{d+1})}\lesssim \|u_{0}\|_{L^{2}(\mathbb{R}^{d})},
\end{equation}
which is the classical endpoint Strichartz estimate for the Schr\"odinger equation\footnote{We refer the reader to Chapter 11 of \cite{MS} for the details of the argument sketched above.}. Mixed-norm variants of \eqref{Strestimate} are extremely useful in the study of local well-posedness of certain nonlinear dispersive PDEs, but we do not discuss them here\footnote{See \cite{LP} for a more detailed account of these applications.}.

We thank Diogo Oliveira e Silva, Mateus Sousa and the two anonymous referees for bringing to our attention many of the cited references and for the feedback provided, which greatly improved the original version of this paper.

\section{Reduction to a localized bilinear estimate}\label{reductionmodel}

Showing that $\mathcal{E}_{2}$ maps $L^{2}$ to $L^{4+\varepsilon}$ for all $\varepsilon>0$ is equivalent to showing that

$$T(f,g)=\mathcal{E}_{2}(f)\cdot\mathcal{E}_{2}(g)$$
maps $L^{2}\times L^{2}$ to $L^{2+\varepsilon}$ for all $\varepsilon>0$, where now both $f$ and $g$ are supported on $[0,1]^{2}$. 

We start by decomposing $[0,1]^{2}\backslash\Delta$, where $\Delta=\{(x,x):x\in [0,1]\}$, into a union of non-overlapping dyadic cubes whose side-lengths are comparable to their distance to the diagonal $\Delta$. This is known as a \textit{Whitney decomposition} and there are many ways to achieve it. The exact way in which one does it is not crucial to our analysis, but the properties that we just described will play a central role in using lower dimensional phenomena to prove a higher dimensional result, which is our goal. For concreteness, we chose the classical construction from \cite{TVV1} represented in Figure \ref{A} as a reference, even though there are important differences between the roles played by the decomposition in \cite{TVV1} and ours (see Remark \ref{rmk1-feb2123}).

Let $\mathcal{D}^{k}_{[0,1]}$ denote the collection of closed dyadic intervals of the form $[l\cdot 2^{-k},(l+1)\cdot 2^{-k}]$, $k\geq 0$, in $[0,1]$. We say $I$ and $J$ in $\mathcal{D}^{k}_{[0,1]}$ are \textit{close} and write $I\sim J$ if they have the same length and are disjoint, but their parents\footnote{We say that $I^{\ast}$ is the \textit{parent} of $I$ if it is the only dyadic interval that contains $I$ and has twice its length.} $I^{\ast}$ and $J^{\ast}$ are not disjoint. We then have

\begin{equation}\label{whitney-jan2023}
[0,1]^{2}\backslash\Delta = \bigcup_{k}\bigcup_{\substack{I_{1}\sim I_{2} \\ I_{1},I_{2}\in\mathcal{D}^{k}_{[0,1]}}}I_{1}\times I_{2},
\end{equation}
with $|I_{1}|=|I_{2}|\approx d(I_{1},I_{2})$\footnote{$|I|$ denotes the length of the interval $I$, $d(I_{1},I_{2})$ is the distance between $I_{1}$ and $I_{2}$ and $\alpha\approx\beta$ means that there is a universal constant $C>0$ such that $$\frac{1}{C}\cdot|\alpha|\leq |\beta|\leq C\cdot|\alpha|.$$}.

\begin{figure}[h]
  \centering
\captionsetup{font=normalsize,skip=1pt,singlelinecheck=on}
  \includegraphics[scale=.5]{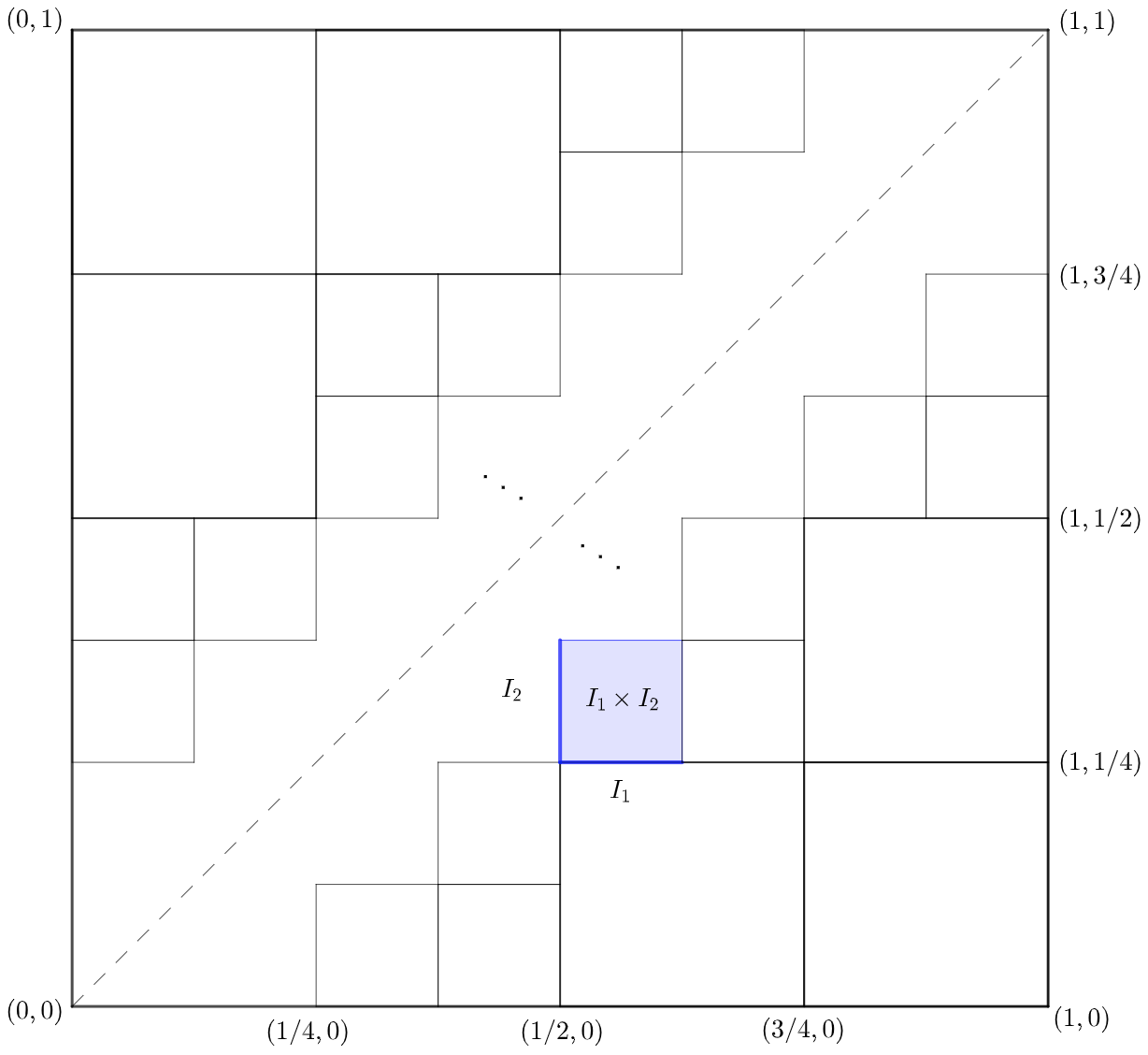}
 \caption{A classical Whitney decomposition.}
 \label{A}
\end{figure}

Decomposing the supports of $f$ and $g$ as above gives us

\begin{equation}\label{ineq1-5nov22}
\eqalign{
\displaystyle T&(f,g)(x,t) \cr
&=\displaystyle \sum_{k_{1},k_{2}\geq 0}\sum_{\substack{ I_{1}\sim I_{2}\in\mathcal{D}_{[0,1]}^{k_{1}} \\ J_{1}\sim J_{2}\in\mathcal{D}_{[0,1]}^{k_{2}}}}\left(\int_{J_{1}}\int_{I_{1}}f(\xi_{1},\xi_{2})e^{-2\pi ix\cdot\xi}e^{-2\pi it|\xi|^{2}}\mathrm{d}\xi_{1}\mathrm{d}\xi_{2}\right)\cdot \left(\int_{J_{2}}\int_{I_{2}}g(\eta_{1},\eta_{2})e^{-2\pi ix\cdot\eta}e^{-2\pi it|\eta|^{2}}\mathrm{d}\eta_{1}\mathrm{d}\eta_{2}\right). \cr
}
\end{equation}

One can further split the union in \eqref{whitney-jan2023} into $O(1)$ non-overlapping unions of cubes $Q\in\mathcal{K}_{\ell}$

\begin{equation*}
\bigcup_{k}\bigcup_{\substack{I_{1}\sim I_{2} \\ I_{1},I_{2}\in\mathcal{D}^{k}_{[0,1]}}}I_{1}\times I_{2} = \bigcup_{O(1)\textnormal{ indices }\ell}\left(\bigcup_{Q\in\mathcal{K}_{\ell}} Q\right)
\end{equation*}
such that the following holds: if $Q=I_{1}\times I_{2}\in\mathcal{K}_{\ell}$, then there is no other interval $I_{2}^{\prime}\neq I_{2}$ with $I_{1}\times I_{2}^{\prime}\in\mathcal{K}_{\ell}$ (equivalently, $I_{1}$ \textit{determines} $I_{2}$ in $\mathcal{K}_{\ell}$). In the picture above, one can see that four such sub-collections $\mathcal{K}_{\ell}$ are enough (we just need to ``separate the four diagonals" of cubes of a given scale). Because there are only $O(1)$ such $\mathcal{K}_{\ell}$, we will assume without loss of generality that the decompositions used in \eqref{ineq1-5nov22} also have this property. In other words, for each $I_{1}$ we will suppose that there is only one $I_{2}$ such that $I_{1}\sim I_{2}$.

\begin{remark}\label{rmk1-feb2123} It is common in the context of the Fourier extension problem to exploit the properties of certain Whitney decompositions, notably in Tao, Vargas and Vega's paper \cite{TVV1}. We highlight a fundamental difference between the employment of this tool in that work and in the present one: in \cite{TVV1}, to prove that certain bilinear extension estimates imply linear ones, one considers a Whitney decomposition for the set
$$U\times U\backslash\widetilde{\Delta}, $$
where $U\subset\mathbb{R}^{d}$ is a cube where both $f$ and $g$ are supported (similarly to the setting we had in the beginning of this section) and $\widetilde{\Delta}:=\{(v,w)\in U\times U:v=w\}$. In other words, assuming $d=2$ to place the discussion in the setting of this paper, \cite{TVV1} performs \textit{a single} decomposition in a \textit{four-dimensional} region. In contrast, we perform \textit{two independent decompositions} in distinct \textit{two-dimensional} regions.
\end{remark}

Fix $t$ and multiply $f(\xi)e^{-2\pi it|\xi|^{2}}$ by a positive bump $\widetilde{\varphi}_{I_{1}\times J_{1}}$ that is $\equiv 1$ on $I_{1}\times J_{1}$ and supported on a slight enlargement of that rectangle. By expanding $f(\xi)e^{-2\pi it|\xi|^{2}}\widetilde{\varphi}_{I_{1}\times J_{1}}(\xi)$ in Fourier series, we obtain:

\begin{equation*}
    f(\xi)e^{-2\pi it|\xi|^{2}}\widetilde{\varphi}_{I_{1}\times J_{1}}(\xi) = \sum_{\overrightarrow{\textit{\textbf{n}}}\in\mathbb{Z}^{2}}\langle f(\cdot)e^{-2\pi it|\cdot|^{2}},\varphi_{I_{1}\times J_{1}}^{\overrightarrow{\textit{\textbf{n}}}}\rangle\varphi_{I_{1}\times J_{1}}^{\overrightarrow{\textit{\textbf{n}}}},
\end{equation*}
where 

$$\varphi_{I_{1}\times J_{1}}^{\overrightarrow{\textit{\textbf{n}}}}(\xi) := \frac{1}{|I_{1}\times J_{1}|^{\frac{1}{2}}}\cdot \varphi_{I_{1}\times J_{1}}(\xi)\cdot e^{-2\pi i \overrightarrow{\textit{\textbf{n}}}\cdot (2^{k_{1}}\xi_{1},2^{k_{2}}\xi_{2})}$$
and $\varphi_{I_{1}\times J_{1}}$ is a bump adapted to $I_{1}\times J_{1}$\footnote{Notice that the bumps $\varphi_{I_{1}\times J_{1}}^{\overrightarrow{\textit{\textbf{n}}}}$ are \textit{$L^{2}$-normalized}. It is \textit{not} the case that for $\overrightarrow{\textit{\textbf{n}}}=0$ one has $\varphi_{I_{1}\times J_{1}}^{0}=\varphi_{I_{1}\times J_{1}}$. The latter is $L^{\infty}$-normalized.}. Analogously,
\begin{equation*}
    g(\eta)e^{-2\pi it|\eta|^{2}}\widetilde{\varphi}_{I_{2}\times J_{2}}(\xi) = \sum_{\overrightarrow{\textit{\textbf{n}}}\in\mathbb{Z}^{2}}\langle g(\cdot)e^{-2\pi it|\cdot|^{2}},\varphi_{I_{2}\times J_{2}}^{\overrightarrow{\textit{\textbf{n}}}}\rangle\varphi_{I_{2}\times J_{2}}^{\overrightarrow{\textit{\textbf{n}}}}.
\end{equation*}

Plugging this in \eqref{ineq1-5nov22},

\begin{equation}\label{ineq2-5nov22}
\eqalign{
\displaystyle T&\displaystyle (f,g)(x,t) \cr
&=\displaystyle \sum_{k_{1},k_{2}\geq 0}\sum_{\substack{ I_{1}\sim I_{2}\in\mathcal{D}_{[0,1]}^{k_{1}} \\ J_{1}\sim J_{2}\in\mathcal{D}_{[0,1]}^{k_{2}}}}\sum_{\overrightarrow{\textit{\textbf{n}}_{1}},\overrightarrow{\textit{\textbf{n}}_{2}}\in\mathbb{Z}^{2}} \langle f(\cdot)e^{-2\pi it|\cdot|^{2}},\varphi_{I_{1}\times J_{1}}^{\overrightarrow{\textit{\textbf{n}}_{1}}}\rangle\cdot \langle g(\cdot)e^{-2\pi it|\cdot|^{2}},\varphi_{I_{2}\times J_{2}}^{\overrightarrow{\textit{\textbf{n}}_{2}}}\rangle \varphi^{I_{1}\times J_{1}}_{\overrightarrow{\textit{\textbf{n}}_{1}}}(x)\cdot \varphi^{I_{2}\times J_{2}}_{\overrightarrow{\textit{\textbf{n}}_{2}}}(x) \cr
&\displaystyle=\sum_{\#=0}^{\infty}\sum_{k_{1},k_{2}\geq 0}\sum_{\substack{ I_{1}\sim I_{2}\in\mathcal{D}_{[0,1]}^{k_{1}} \\ J_{1}\sim J_{2}\in\mathcal{D}_{[0,1]}^{k_{2}}}}\sum_{\|\overrightarrow{\textit{\textbf{n}}_{1}}-\overrightarrow{\textit{\textbf{n}}_{2}}\|_{1}=\#} \langle f(\cdot)e^{-2\pi it|\cdot|^{2}},\varphi_{I_{1}\times J_{1}}^{\overrightarrow{\textit{\textbf{n}}_{1}}}\rangle\cdot \langle g(\cdot)e^{-2\pi it|\cdot|^{2}},\varphi_{I_{2}\times J_{2}}^{\overrightarrow{\textit{\textbf{n}}_{2}}}\rangle \varphi^{I_{1}\times J_{1}}_{\overrightarrow{\textit{\textbf{n}}_{1}}}(x)\cdot \varphi^{I_{2}\times J_{2}}_{\overrightarrow{\textit{\textbf{n}}_{2}}}(x) \cr
&\displaystyle=\sum_{\#=0}^{\infty}T_{\#}(f,g)(x,t), \cr
}
\end{equation}
where
\begin{equation*}
\eqalign{
    T_{\#}&\displaystyle (f,g)(x,t) \cr
    &\displaystyle = \sum_{k_{1},k_{2}\geq 0}\sum_{\substack{ I_{1}\sim I_{2}\in\mathcal{D}_{[0,1]}^{k_{1}} \\ J_{1}\sim J_{2}\in\mathcal{D}_{[0,1]}^{k_{2}}}}\sum_{\|\overrightarrow{\textit{\textbf{n}}_{1}}-\overrightarrow{\textit{\textbf{n}}_{2}}\|_{1}=\#} \langle f(\cdot)e^{-2\pi it|\cdot|^{2}},\varphi_{I_{1}\times J_{1}}^{\overrightarrow{\textit{\textbf{n}}_{1}}}\rangle\cdot \langle g(\cdot)e^{-2\pi it|\cdot|^{2}},\varphi_{I_{2}\times J_{2}}^{\overrightarrow{\textit{\textbf{n}}_{2}}}\rangle \varphi^{I_{1}\times J_{1}}_{\overrightarrow{\textit{\textbf{n}}_{1}}}(x)\cdot \varphi^{I_{2}\times J_{2}}_{\overrightarrow{\textit{\textbf{n}}_{2}}}(x) \cr
    }
\end{equation*}
and\footnote{We highlight the subtlety of the notation $\varphi^{Q}_{\overrightarrow{\textit{\textbf{n}}}}$; we swapped the indices $Q$ and $\overrightarrow{\textit{\textbf{n}}}$.}
$$\varphi^{Q}_{\overrightarrow{\textit{\textbf{n}}}}(x) := \widehat{\varphi_{Q}^{\overrightarrow{\textit{\textbf{n}}}}}(x). $$

It is important to remark that since $\varphi^{Q}_{\overrightarrow{\textit{\textbf{n}}}}$ is the Fourier transform of an $L^{2}$-normalized bump, it is also $L^{2}$-normalized. Observe that if $\overrightarrow{\textit{\textbf{n}}_{1}}\neq \overrightarrow{\textit{\textbf{n}}_{2}}$, the bumps $\varphi^{I_{1}\times J_{1}}_{\overrightarrow{\textit{\textbf{n}}_{1}}}$ and $\varphi^{I_{2}\times J_{2}}_{\overrightarrow{\textit{\textbf{n}}_{2}}}$ are essentially supported in different cubes, so the product
\begin{equation}\label{disjointsupp-6nov22}
    \varphi^{I_{1}\times J_{1}}_{\overrightarrow{\textit{\textbf{n}}_{1}}}(x)\cdot \varphi^{I_{2}\times J_{2}}_{\overrightarrow{\textit{\textbf{n}}_{2}}}(x) 
\end{equation}
decays very fast as $\#$ grows. We will prove bounds for the $\#=0$ term, henceforth denoted by $\widetilde{T}$. The argument is the same for any other $T_{\#}$ and it comes with an operatorial bound of $O(|\#|^{-100})$ due to the fast decay of the bumps in \eqref{disjointsupp-6nov22}, which is enough to sum the series in $\#$.

Finally, we discretize the $t$ variable. Multiply both sides of the expression of $\widetilde{T}$ by
$$1=\sum_{m\in\mathbb{Z}}\chi_{m}, $$
where $\chi_{m}(t)$ is the indicator function of the interval $[m,m+1)$. This gives

\begin{equation*}\label{disct-6nov22}
\eqalign{
\displaystyle \widetilde{T}&(f,g)(x,t) \cr
&=\displaystyle \sum_{k_{1},k_{2}\geq 0}\sum_{\substack{ I_{1}\sim I_{2}\in\mathcal{D}_{[0,1]}^{k_{1}} \\ J_{1}\sim J_{2}\in\mathcal{D}_{[0,1]}^{k_{2}}}}\sum_{(\overrightarrow{\textit{\textbf{n}}},m)\in\mathbb{Z}^{3}} \langle f(\cdot)e^{-2\pi it|\cdot|^{2}},\varphi_{I_{1}\times J_{1}}^{\overrightarrow{\textit{\textbf{n}}}}\rangle\cdot \langle g(\cdot)e^{-2\pi it|\cdot|^{2}},\varphi_{I_{2}\times J_{2}}^{\overrightarrow{\textit{\textbf{n}}}}\rangle \varphi^{I_{1}\times J_{1}}_{\overrightarrow{\textit{\textbf{n}}}}(x)\cdot \varphi^{I_{2}\times J_{2}}_{\overrightarrow{\textit{\textbf{n}}}}(x)\cdot \chi_{m}(t), \cr
}
\end{equation*}

For a fixed $t$, one can write
\begin{equation*}
    \eqalign{
    e^{-2\pi it|\xi|^{2}}\cdot\varphi_{[0,1]^{2}}(\xi)&\displaystyle = e^{-2\pi im|\xi|^{2}}\cdot e^{-2\pi i(t-m)|\xi|^{2}}\varphi_{[0,1]^{2}}(\xi) \cr
    &\displaystyle= e^{-2\pi im|\xi|^{2}}\cdot\sum_{\overrightarrow{\textit{\textbf{u}}}\in\mathbb{Z}^{2}}\langle e^{-2\pi i(t-m)|\cdot|^{2}},\varphi_{[0,1]^{2}}^{\overrightarrow{\textit{\textbf{u}}}}\rangle\varphi_{[0,1]^{2}}^{\overrightarrow{\textit{\textbf{u}}}} \cr
    }
\end{equation*}
by expanding $e^{-2\pi i(t-m)|\xi|^{2}}$ as a Fourier series\footnote{$\varphi_{[0,1]^{2}}$ denotes a bump that is $\equiv 1$ on $[0,1]^{2}$ and is supported in a slightly bigger box.} at scale $1$. Therefore we can write $\widetilde{T}(f,g)$ as

\begin{equation*}
\eqalign{
&\displaystyle \sum_{k_{1},k_{2}\geq 0}\sum_{\substack{ I_{1}\sim I_{2}\in\mathcal{D}_{[0,1]}^{k_{1}} \\ J_{1}\sim J_{2}\in\mathcal{D}_{[0,1]}^{k_{2}}}}\sum_{(\overrightarrow{\textit{\textbf{n}}},m)\in\mathbb{Z}^{3}}\sum_{\substack{\overrightarrow{\textit{\textbf{u}}}\in\mathbb{Z}^{2} \\ \overrightarrow{\textit{\textbf{v}}}\in\mathbb{Z}^{2}}} \cr
&\qquad\displaystyle \langle f(\cdot)e^{-2\pi im|\cdot|^{2}},\varphi_{I_{1}\times J_{1}}^{\overrightarrow{\textit{\textbf{n}}}}\cdot \varphi_{[0,1]^{2}}^{\overrightarrow{\textit{\textbf{u}}}}\rangle\cdot \langle g(\cdot)e^{-2\pi it|\cdot|^{2}},\varphi_{I_{2}\times J_{2}}^{\overrightarrow{\textit{\textbf{n}}}}\cdot \varphi_{[0,1]^{2}}^{\overrightarrow{\textit{\textbf{v}}}}\rangle \varphi^{I_{1}\times J_{1}}_{\overrightarrow{\textit{\textbf{n}}}}(x)\cdot \varphi^{I_{2}\times J_{2}}_{\overrightarrow{\textit{\textbf{n}}}}(x)\cdot C_{\overrightarrow{\textit{\textbf{u}}}}^{m,t}\cdot C_{\overrightarrow{\textit{\textbf{v}}}}^{m,t}\cdot \chi_{m}(t), \cr
}
\end{equation*}
where
$$C_{\overrightarrow{\textit{\textbf{u}}}}^{m,t}= \langle e^{-2\pi i(t-m)|\cdot|^{2}},\varphi_{[0,1]^{2}}^{\overrightarrow{\textit{\textbf{u}}}}\rangle,$$
$$C_{\overrightarrow{\textit{\textbf{v}}}}^{m,t}= \langle e^{-2\pi i(t-m)|\cdot|^{2}},\varphi_{[0,1]^{2}}^{\overrightarrow{\textit{\textbf{v}}}}\rangle.$$

For the expression defining $\widetilde{T}$ to be nonzero, $m$ must satisfy $|t-m|\leq 1$, hence the Fourier coefficients $C_{\overrightarrow{\textit{\textbf{u}}}}^{m,t}$ and $C_{\overrightarrow{\textit{\textbf{v}}}}^{m,t}$ decay like $O(|\overrightarrow{\textit{\textbf{u}}}|^{-100})$ and $O(|\overrightarrow{\textit{\textbf{v}}}|^{-100})$, respectively. In addition, the extra terms
$$\varphi_{[0,1]^{2}}^{\overrightarrow{\textit{\textbf{u}}}}\quad\textnormal{and}\quad\varphi_{[0,1]^{2}}^{\overrightarrow{\textit{\textbf{v}}}}$$
in the scalar products involving $f$ and $g$, respectively, simply shift the integrands in frequency, and this does not affect in any way the arguments that follow.

It is then enough to treat the $\overrightarrow{\textit{\textbf{u}}}=\overrightarrow{\textit{\textbf{v}}}=0$ case and the model operator (still detoned by $\widetilde{T}$ by a slight abuse of notation\footnote{There is another slight abuse of notation here: observe that $\widetilde{\chi}_{m}(t):=C_{\overrightarrow{\textit{\textbf{0}}}}^{m,t}\cdot\chi_{m}(t)$ is a smooth function supported in $[m,m+1)$, which is all that is needed in the proof. To simplify the notation, we will continue to call it $\chi_{m}(t)$ to lighten the notation.}):
\begin{equation*}\label{ineq3-5nov22}
\eqalign{
\displaystyle \widetilde{T}&(f,g)(x,t) \cr
&=\displaystyle \sum_{k_{1},k_{2}\geq 0}\sum_{\substack{ I_{1}\sim I_{2}\in\mathcal{D}_{[0,1]}^{k_{1}} \\ J_{1}\sim J_{2}\in\mathcal{D}_{[0,1]}^{k_{2}}}}\sum_{(\overrightarrow{\textit{\textbf{n}}},m)\in\mathbb{Z}^{3}} \langle f(\cdot)e^{-2\pi im|\cdot|^{2}},\varphi_{I_{1}\times J_{1}}^{\overrightarrow{\textit{\textbf{n}}}}\rangle\cdot \langle g(\cdot)e^{-2\pi im|\cdot|^{2}},\varphi_{I_{2}\times J_{2}}^{\overrightarrow{\textit{\textbf{n}}}}\rangle \varphi^{I_{1}\times J_{1}}_{\overrightarrow{\textit{\textbf{n}}}}(x)\cdot \varphi^{I_{2}\times J_{2}}_{\overrightarrow{\textit{\textbf{n}}}}(x)\cdot \chi_{m}(t). \cr
}
\end{equation*}

Fix $k_{1},k_{2}\geq 0$ and consider

\begin{equation*}
\eqalign{
    \widetilde{T}_{k_{1},k_{2}}&(f,g)(x,t) \cr
    &\displaystyle = \sum_{\substack{ I_{1}\sim I_{2}\in\mathcal{D}_{[0,1]}^{k_{1}} \\ J_{1}\sim J_{2}\in\mathcal{D}_{[0,1]}^{k_{2}}}}\sum_{(\overrightarrow{\textit{\textbf{n}}},m)\in\mathbb{Z}^{3}} \langle f(\cdot)e^{-2\pi im|\cdot|^{2}},\varphi_{I_{1}\times J_{1}}^{\overrightarrow{\textit{\textbf{n}}}}\rangle\cdot \langle g(\cdot)e^{-2\pi im|\cdot|^{2}},\varphi_{I_{2}\times J_{2}}^{\overrightarrow{\textit{\textbf{n}}}}\rangle \varphi^{I_{1}\times J_{1}}_{\overrightarrow{\textit{\textbf{n}}}}(x)\cdot \varphi^{I_{2}\times J_{2}}_{\overrightarrow{\textit{\textbf{n}}}}(x)\cdot \chi_{m}(t).
    }
\end{equation*}

To show that $\widetilde{T}$ maps $L^{p}\times L^{p}\mapsto L^{q}$ (and hence the same for $T$) it suffices to prove that each $\widetilde{T}_{k_{1},k_{2}}$ also maps $L^{p}\times L^{p}\mapsto L^{q}$ with operatorial norm $O(2^{-c_{p,q}(k_{1}+k_{2})})$, where $c_{p,q}$ is a positive constant.
This is is content of Theorem \ref{mainprop} in Section \ref{main}.

\section{The main building block}

The results of \cite{MO} are obtained from three fundamental building blocks, as explained in Section 4 of that paper; they allow us to acquire higher dimensional information from lower dimensional phenomena. The proof of Theorem \ref{mainprop} relies on a one-dimensional bound similar to Proposition 4.4 of \cite{MO}.

Let $h_{1}$ and $h_{2}$ be functions supported on $[0,1]$. To keep our notation consistent with the one from the previous section, let $\varphi_{I}$ denote a bump adapted to the interval $I$ and define
$$\varphi_{I}^{n}(\xi):=\frac{1}{|I|^{\frac{1}{2}}}\cdot\varphi_{I}(\xi)\cdot e^{-2\pi in|I|^{-1}\xi}.$$

We have the following bound:

\begin{proposition}\label{bilinear1d-6nov22} For a fixed pair $(I_{1},I_{2})$ with $2^{-k}=|I_{1}|=|I_{2}|\approx d(I_{1},I_{2})$,

\begin{equation*}
S= \sum_{n,m\in\mathbb{Z}} |\langle h_{1}(\cdot)e^{-2\pi im|\cdot|^{2}},\varphi_{I_{1}}^{n}\rangle|^{2}\cdot |\langle h_{2}(\cdot)e^{-2\pi im|\cdot|^{2}},\varphi_{I_{2}}^{n}\rangle|^{2} \lesssim\frac{1}{2^{-2k}}\cdot \|h_{1}\cdot\varphi_{I_{1}}\|_{2}^{2}\cdot \|h_{2}\cdot\varphi_{I_{2}}\|_{2}^{2}.
\end{equation*}
\end{proposition}
\begin{proof} Define
\begin{equation*}
\eqalign{
\Phi^{n,m}_{I_{1},I_{2}}(\xi,\eta) &\displaystyle := e^{-2\pi im\xi^{2}}\cdot e^{2\pi im\eta^{2}}\cdot\varphi_{I_{1}}^{n}(\xi)\cdot\overline{\varphi_{I_{2}}^{n}(\eta)} \cr
&\displaystyle = e^{-2\pi im\xi^{2}}\cdot e^{2\pi im\eta^{2}}\cdot \left(\frac{1}{|I_{1}|^{\frac{1}{2}}}\cdot\varphi_{I_{1}}(\xi)\cdot e^{-2\pi in2^{k}\xi}\right)\cdot \left(\frac{1}{|I_{2}|^{\frac{1}{2}}}\cdot\varphi_{I_{2}}(\eta)\cdot e^{2\pi in2^{k}\eta}\right) \cr
&\displaystyle = \frac{1}{2^{-k}}\cdot e^{-2\pi i 2^{k}(\xi-\eta)}\cdot e^{-2\pi im(\xi-\eta)(\xi+\eta)}\cdot\varphi_{I_{1}}(\xi)\cdot\varphi_{I_{2}}(\eta). \cr
}
\end{equation*}

Hence $S$ can be rewritten as
\begin{equation*}
S=\sum_{n,m\in\mathbb{Z}} |\langle h_{1}\otimes h_{2},\Phi^{n,m}_{I_{1},I_{2}}\rangle|^{2}.
\end{equation*}

Observe that by performing the change of variables $\alpha=\xi-\eta$, $\beta=\xi+\eta$ we obtain $|\alpha|\approx 2^{-k}$, the region $I_{1}\times I_{2}$ gets mapped (linearly) to a compact $U\subset\mathbb{R}^{2}$ and
\begin{equation*}
    \eqalign{
    \displaystyle|&\displaystyle\langle h_{1}\otimes h_{2}, \Phi^{n,m}_{I_{1},I_{2}}\rangle|^{2} \cr &\displaystyle= \left|\iint_{I_{1}\times I_{2}} h_{1}(\xi)h_{2}(\eta)\cdot\frac{1}{2^{-k}}\cdot e^{2\pi i n2^{k}(\xi-\eta)}\cdot e^{2\pi im(\xi-\eta)(\xi+\eta)}\cdot\varphi_{I_{1}}(\xi)\cdot\varphi_{I_{2}}(\eta)\mathrm{d}\xi\mathrm{d}\eta\right|^{2} \cr
    &\displaystyle \approx \left|\iint_{U} h_{1}\left(\frac{\beta+\alpha}{2}\right)h_{2}\left(\frac{\beta-\alpha}{2}\right)\cdot\frac{1}{2^{-k}}\cdot e^{2\pi i n2^{k}\alpha}\cdot e^{2\pi im\alpha\beta}\cdot\varphi_{I_{1}}\left(\frac{\beta+\alpha}{2}\right)\cdot\varphi_{I_{2}}\left(\frac{\beta-\alpha}{2}\right)\mathrm{d}\alpha\mathrm{d}\beta\right|^{2} \cr
     &\displaystyle \approx \frac{1}{2^{-k}}\left|\iint_{U} h_{1}\left(\frac{\beta+\alpha}{2}\right)h_{2}\left(\frac{\beta-\alpha}{2}\right)\cdot\frac{1}{2^{-\frac{k}{2}}}\cdot e^{2\pi in 2^{k}\alpha}\cdot \widetilde{\varphi}(\alpha) \cdot e^{2\pi im\alpha\beta}\cdot\varphi_{I_{1}}\left(\frac{\beta+\alpha}{2}\right)\cdot\varphi_{I_{2}}\left(\frac{\beta-\alpha}{2}\right)\mathrm{d}\alpha\mathrm{d}\beta\right|^{2}, \cr
    }
\end{equation*}
where we inserted a bump $\widetilde{\varphi}$ that is $1$ on the support of the integrand in $\alpha$, but supported on a fixed dilate of it. This is done so we can fit $U$ inside a dilate\footnote{If $c>0$, $cI$ denotes the interval that has the same center as $I$ and size $c|I|$.} $cI_{1}\times cI_{2}$ (with $c$ independent of $I_{1}$ and $I_{2}$), which is possible since $|\alpha|,|\beta|\approx 2^{-k}$. The family of bumps given by
$$\widetilde{\varphi}^{n}(\alpha):=\frac{1}{2^{-\frac{k}{2}}}\cdot e^{2\pi in 2^{k}\alpha}\cdot \widetilde{\varphi}(\alpha) $$
is $L^{2}$-normalized (up to a universal constant independent of $k$ and $I$). Define

$$H(\alpha,\beta):= h_{1}\left(\frac{\beta+\alpha}{2}\right)h_{2}\left(\frac{\beta-\alpha}{2}\right) \varphi_{I_{1}}\left(\frac{\beta+\alpha}{2}\right)\varphi_{I_{2}}\left(\frac{\beta-\alpha}{2}\right)$$
and, by using Bessel twice (in $n$ and in $m$, as we did in Section 7 of \cite{MO}), we conclude that
\begin{equation*}
\eqalign{
\displaystyle\sum_{n,m\in\mathbb{Z}^{2}} |\langle h_{1}\otimes h_{2},\Phi^{n,m}_{I_{1},I_{2}}\rangle|^{2} &\displaystyle\lesssim \frac{1}{2^{-k}}\cdot \sum_{m\in\mathbb{Z}}\sum_{n\in\mathbb{Z}}\left|\int_{cI_{1}}\left(\int_{cI_{2}} H(\alpha,\beta) \cdot e^{2\pi im\alpha\beta}\mathrm{d}\beta\right)\cdot \widetilde{\varphi}^{n}(\alpha)\mathrm{d}\alpha\right|^{2} \cr
&\lesssim\displaystyle \frac{1}{2^{-k}}\cdot\frac{1}{2^{-k}}\cdot\|H\|_{2}^{2} \cr
&\lesssim\displaystyle \frac{1}{2^{-2k}}\cdot \|h_{1}\cdot\varphi_{I_{1}}\|_{2}^{2}\cdot \|h_{2}\cdot\varphi_{I_{2}}\|_{2}^{2}, \cr
}
\end{equation*}
where this extra $\frac{1}{2^{-k}}$ factor comes from applying Bessel when summing in $m$ since $|\alpha|\approx 2^{-k}$.
\end{proof}

\section{Proof of the main theorem}\label{main}
For convenience of the reader, we start by briefly recalling and adapting the notation from our earlier work \cite{MO}. A few new objects particularly related to the problem at hand are also introduced.

\begin{itemize}
\item The index $Q$ (that could be either an interval or a rectangle) in $\varphi_{Q}$ and $\widetilde{\varphi}_{Q}$ indicates that these are \textit{$L^{\infty}$-normalized} bumps adapted to $Q$; in other words, they are $\equiv 1$ on $Q$, smooth and supported on a slight enlargement of $Q$.
\item The bumps $\varphi_{I}^{n}$ with a superscript $n$ are \textit{$L^{2}$-normalized} and given by
$$\varphi_{I}^{n}(\xi):=\frac{1}{|I|^{\frac{1}{2}}}\cdot\varphi_{I}(\xi)\cdot e^{-2\pi in|I|^{-1}\xi}.$$

Their two-dimensional analogues are
$$\varphi_{I_{1}\times J_{1}}^{\overrightarrow{\textit{\textbf{n}}}}(\xi) := \frac{1}{|I_{1}\times J_{1}|^{\frac{1}{2}}}\cdot \varphi_{I_{1}\times J_{1}}(\xi)\cdot e^{-2\pi i \overrightarrow{\textit{\textbf{n}}}\cdot (2^{k_{1}}\xi_{1},2^{k_{2}}\xi_{2})}.$$.
\item We swap $Q$ and $\overrightarrow{\textit{\textbf{n}}}$ to indicate the Fourier transform of $\varphi_{I_{1}\times J_{1}}^{\overrightarrow{\textit{\textbf{n}}}}$:
$$\varphi^{Q}_{\overrightarrow{\textit{\textbf{n}}}}(x) := \widehat{\varphi_{Q}^{\overrightarrow{\textit{\textbf{n}}}}}(x). $$
\item The index $\xi_{i}$ in $\langle \cdot,\cdot\rangle_{\xi_{i}}$ indicates that the inner product is an integral in the variable $\xi_{i}$ only. For instance,
\begin{equation}\label{ex1-211221}
\langle f,\varphi\rangle_{\xi_{1}} := \int_{\mathbb{R}}f(\xi_{1},\xi_{2})\cdot\overline{\varphi}(\xi_{1})\mathrm{d}\xi_{1}
\end{equation}
is now a function of the variables $\xi_{2}$.
\item The upper index in $f^{\xi_{i}}$ indicates that the variable $\xi_{i}$ is fixed. For instance, if $f$ is a function of $(\xi_{1},\xi_{2})$, in
$$\langle f^{\xi_{1}},\varphi\rangle_{\xi_{2}} $$
the scalar product is an integral in $\xi_{2}$ only, hence $\xi_{1}$ is fixed.

\item The expression $\left\|\langle f,\cdot\rangle_{\xi_{i}}\right\|_{2} $ is the $L^{2}$ norm of a function in the variable $\xi_{l}$, $l\neq i$. To illustrate using \eqref{ex1-211221},
$$\|\langle f,\varphi\rangle_{\xi_{1}}\|_{2} = \left[\int_{\mathbb{R}}\left|\int_{\mathbb{R}}f(\xi_{1},\xi_{2})\cdot\overline{\varphi}(\xi_{1})\mathrm{d}\xi_{1}\right|^{2}\mathrm{d}\xi_{2}\right]^{\frac{1}{2}}.
 $$
 \item We will work on the space $\Omega = \mathbb{Z}^{3}\times\mathcal{D}^{k_{1}}_{[0,1]}\times \mathcal{D}^{k_{2}}_{[0,1]}\times\mathcal{D}^{k_{1}}_{[0,1]}\times \mathcal{D}^{k_{2}}_{[0,1]}$. Because of the properties of the Whitney decomposition from Section \ref{reductionmodel}, the subset of $\Omega$ that will appear in this section depends on $5$ parameters, not 7; this is because a pair $(I_{1},J_{1})\in\mathcal{D}^{k_{1}}_{[0,1]}\times \mathcal{D}^{k_{2}}_{[0,1]}$ determines the $(I_{2},J_{2})\in\mathcal{D}^{k_{1}}_{[0,1]}\times \mathcal{D}^{k_{2}}_{[0,1]}$ (and vice-versa) that appears in any vector $(n_{1},n_{2},m,I_{1},J_{1},I_{2},J_{2})\in\Omega$. In that context, if $A\subset\Omega$ and $\mathbbm{1}_{A}$ denotes the indicator function of $A$, we define:
 
 $$\|\mathbbm{1}_{A}\|_{\ell^{\infty}_{n_{1},m,I_{2}}\ell^{1}_{n_{2},J_{2}}} := \sup_{(n_{1},m,I_{2})\in\mathbb{Z}^{2}\times\mathcal{D}^{k_{1}}_{[0,1]}}\sum_{(n_{2},J_{2})\in\mathbb{Z}\times\mathcal{D}^{k_{2}}_{[0,1]}}\mathbbm{1}_{A}(n_{1},n_{2},m,I_{1},J_{1},I_{2},J_{2}).$$
 
 Notice that the right-hand side does not depend on any free parameter by our previous observation. The quantity $\|\mathbbm{1}_{A}\|_{\ell^{\infty}_{n_{2},m,J_{2}}\ell^{1}_{n_{1},J_{1}}}$ is defined analogously.
\end{itemize}

As explained in the end of Section \ref{reductionmodel}, the following bound implies Theorem \ref{mainthm}:

\begin{theorem}\label{mainprop} Given $\varepsilon>0$,
$$\|\widetilde{T}_{k_{1},k_{2}}(f,g)\|_{2+\varepsilon^{\prime}}\lesssim_{\varepsilon} 2^{-\frac{\varepsilon}{3} (k_{1}+k_{2})}\|f\|_{2}\cdot\|g\|_{2}, $$
where $\varepsilon^{\prime}=\frac{8\varepsilon}{1-2\varepsilon}$.
\end{theorem}

\begin{proof}[Proof of Theorem \ref{mainprop}] In what follows, let $E_{1}\subset Q_{1}$,$E_{2}\subset Q_{2}$, and $F\subset\mathbb{R}^{d+1}$ be measurable sets such that $|f|\leq\chi_{E_{1}}$, $|g|\leq\chi_{E_{2}}$ and $|H|\leq\chi_{F}$. We define the associated trilinear form obtained by dualizing $\widetilde{T}_{k_{1},k_{2}}$:

\begin{equation}\label{form-5nov22}
\eqalign{
\widetilde{\Lambda}_{k_{1},k_{2}}&(f,g,H) \cr
&\displaystyle:= \sum_{\substack{ I_{1}\sim I_{2}\in\mathcal{D}_{[0,1]}^{k_{1}} \\ J_{1}\sim J_{2}\in\mathcal{D}_{[0,1]}^{k_{2}}}}\sum_{(\overrightarrow{\textit{\textbf{n}}},m)\in\mathbb{Z}^{3}} \langle f(\cdot)e^{-2\pi im|\cdot|^{2}},\varphi_{I_{1}\times J_{1}}^{\overrightarrow{\textit{\textbf{n}}}}\rangle\cdot \langle g(\cdot)e^{-2\pi im|\cdot|^{2}},\varphi_{I_{2}\times J_{2}}^{\overrightarrow{\textit{\textbf{n}}}}\rangle \cdot\langle H,\varphi^{I_{1}\times J_{1}}_{\overrightarrow{\textit{\textbf{n}}}}\cdot \varphi^{I_{2}\times J_{2}}_{\overrightarrow{\textit{\textbf{n}}}}\otimes \chi_{m}\rangle. \cr
}
\end{equation}

Interpolation theory shows that to prove the theorem above it suffices to show that
\begin{equation}\label{restweaktypebdd}
|\Lambda_{k_{1},k_{2}}(f,g,H)|\lesssim_{\varepsilon}2^{-\frac{\varepsilon}{4} (k_{1}+k_{2})} |E_{1}|^{\gamma_{1}}\cdot |E_{2}|^{\gamma_{2}}\cdot |F|^{\gamma_{3}} 
\end{equation}
for all $\varepsilon>0$, where $\gamma_{j}$ $(1\leq j\leq 2)$ and $\gamma_{3}$ are in a small neighborhood of $\frac{1}{2}$ and $\frac{1}{2}+\varepsilon$, respectively\footnote{We refer the reader to Chapter 3 of \cite{Thiele1} for a detailed account of multilinear interpolation theory.}. To keep the notation simple, the restricted weak-type estimate that we will prove will be for the centers of such neighborhoods. In other words, we will show that

\begin{equation*}
|\Lambda_{k_{1},k_{2}}(f,g,H)|\lesssim_{\varepsilon}2^{-\frac{\varepsilon}{4} (k_{1}+k_{2})} |E_{1}|^{\frac{1}{2}}\cdot |E_{2}|^{\frac{1}{2}}\cdot |F|^{\frac{1}{2}+\frac{2\varepsilon}{1+2\varepsilon}}
\end{equation*}
for all $\varepsilon>0$, but it will be clear from the arguments that as long as we give this $\varepsilon>0$ away, a slightly different choice of interpolation parameters yields \eqref{restweaktypebdd}.

We will define several level sets that encode the sizes of many quantities that will play a role in the proof. We start with the ones involving the scalar products in the trilinear form above.

$$\mathbb{A}^{l_{1}}_{1}=\left\{(\overrightarrow{\textit{\textbf{n}}},m, I_{1},J_{1})\in \mathbb{Z}^{3}\times \mathcal{D}^{k_{1}}_{[0,1]}\times \mathcal{D}^{k_{2}}_{[0,1]} : |\langle f(\cdot)e^{-2\pi im|\cdot|^{2}},\varphi_{I_{1}\times J_{1}}^{\overrightarrow{\textit{\textbf{n}}}}\rangle| \approx 2^{-l_{1}}\right\},$$
$$\mathbb{A}^{l_{2}}_{2}=\left\{(\overrightarrow{\textit{\textbf{n}}},m, I_{2},J_{2})\in \mathbb{Z}^{3}\times \mathcal{D}^{k_{1}}_{[0,1]}\times \mathcal{D}^{k_{2}}_{[0,1]} : |\langle g(\cdot)e^{-2\pi im|\cdot|^{2}},\varphi_{I_{2}\times J_{2}}^{\overrightarrow{\textit{\textbf{n}}}}\rangle| \approx 2^{-l_{2}}\right\}. $$

The two-dimensional scalar products above are not the only information that we will need to control. As we will see, some mixed-norm quantities appear naturally after using Bessel's inequality along certain directions, and we will need to capture these as well:

$$\mathbb{B}^{r_{1}}_{1}=\left\{(n_{1},m,I_{1})\in \mathbb{Z}^{2}\times \mathcal{D}^{k_{1}}_{[0,1]}:\left\|\left\langle f^{\xi_{2}}(\cdot)e^{-2\pi im|\cdot|^{2}},\varphi^{n_{1}}_{I_{1}}\right\rangle_{\xi_{1}}\right\|_{L^{2}_{\xi_{2}}}\approx 2^{-r_{1}}\right\},$$

$$\mathbb{B}^{r_{2}}_{2}=\left\{(n_{2},m,J_{1})\in \mathbb{Z}^{2}\times \mathcal{D}^{k_{2}}_{[0,1]}:\left\|\left\langle f^{\xi_{1}}(\cdot)e^{-2\pi im|\cdot|^{2}},\varphi^{n_{2}}_{J_{1}}\right\rangle_{\xi_{2}}\right\|_{L^{2}_{\xi_{1}}}\approx 2^{-r_{2}}\right\},$$

$$\mathbb{C}^{s_{1}}_{1}=\left\{(n_{1},m,I_{2})\in \mathbb{Z}^{2}\times \mathcal{D}^{k_{1}}_{[0,1]}:\left\|\left\langle g^{\xi_{2}}(\cdot)e^{-2\pi im|\cdot|^{2}},\varphi^{n_{1}}_{I_{2}}\right\rangle_{\xi_{1}}\right\|_{L^{2}_{\xi_{2}}}\approx 2^{-s_{1}}\right\},$$

$$\mathbb{C}^{s_{2}}_{2}=\left\{(n_{2},m,J_{2})\in \mathbb{Z}^{2}\times \mathcal{D}^{k_{2}}_{[0,1]}:\left\|\left\langle g^{\xi_{1}}(\cdot)e^{-2\pi im|\cdot|^{2}},\varphi^{n_{2}}_{J_{2}}\right\rangle_{\xi_{2}}\right\|_{L^{2}_{\xi_{1}}}\approx 2^{-s_{2}}\right\}.$$

The last quantity we have to control is the one arising from the dualizing function $H$:
$$\mathbb{D}^{t}=\{(\overrightarrow{\textit{\textbf{n}}},m, I_{1},J_{1},I_{2},J_{2})\in \mathbb{Z}^{3}\times \mathcal{D}^{k_{1}}_{[0,1]}\times \mathcal{D}^{k_{2}}_{[0,1]}\times \mathcal{D}^{k_{1}}_{[0,1]}\times \mathcal{D}^{k_{2}}_{[0,1]}: |\langle H,\varphi^{I_{1}\times J_{1}}_{\overrightarrow{\textit{\textbf{n}}}}\cdot \varphi^{I_{2}\times J_{2}}_{\overrightarrow{\textit{\textbf{n}}}}\otimes \chi_{m}\rangle| \approx 2^{-t}\}. $$

We remark that a fixed $I_{1}$ determines $I_{2}$, as well as a fixed $J_{1}$ determines $J_{2}$, hence the sum in \eqref{form-5nov22} depends on five parameters (two intervals and three integers). This is because the cubes $I_{1}\times I_{2}$ and $J_{1}\times J_{2}$ are assumed to be in a fixed strip, as one can see from the previous Whitney decomposition.

In order to prove some crucial bounds that will play an important role later on, we will have to isolate the previous information for only one of the functions $f$ and $g$. This will be done in terms of the following sets: 

$$\mathbb{X}^{l_{1},r_{1}}=\mathbb{A}^{l_{1}}_{1}\cap\{(\overrightarrow{\textit{\textbf{n}}},m, I_{1},J_{1}); (n_{1},m,I_{1})\in\mathbb{B}^{r_{1}}_{1}\}, \qquad\mathbb{X}^{l_{1},r_{2}}=\mathbb{A}^{l_{1}}_{1}\cap\{(\overrightarrow{\textit{\textbf{n}}},m,I_{1},J_{1}); (n_{2},m,J_{1})\in\mathbb{B}^{r_{2}}_{2}\}, $$
$$\mathbb{X}^{l_{2},s_{1}}=\mathbb{A}^{l_{2}}_{2}\cap\{(\overrightarrow{\textit{\textbf{n}}},m,I_{2},J_{2}); (n_{1},m,I_{2})\in\mathbb{C}^{s_{1}}_{1}\}, \qquad\mathbb{X}^{l_{2},s_{2}}=\mathbb{A}^{l_{2}}_{2}\cap\{(\overrightarrow{\textit{\textbf{n}}},m,I_{2},J_{2}); (n_{2},m,J_{2})\in\mathbb{C}^{s_{2}}_{2}\}. $$

In other words, the set $\mathbb{X}^{l_{1},r_{1}}$ for instance contains all the $(\overrightarrow{\textit{\textbf{n}}},m,I_{1},J_{1})$ whose corresponding scalar product 
$$\langle f(\cdot)e^{-2\pi im|\cdot|^{2}},\varphi_{I_{1}\times J_{1}}^{\overrightarrow{\textit{\textbf{n}}}}\rangle$$
has size about $2^{-l_{1}}$ and with $(n_{1},m,I_{1})$ being such that 
$$\left\|\left\langle f^{\xi_{2}}(\cdot)e^{-2\pi im|\cdot|^{2}},\varphi^{n_{1}}_{I_{1}}\right\rangle_{\xi_{1}}\right\|_{L^{2}_{\xi_{2}}}$$
has size about $2^{-r_{1}}$.

Finally, it will also be important to encode all the previous information into one single set. This will be done with

$$\mathbb{X}^{\overrightarrow{l},\overrightarrow{r},\overrightarrow{s},t}=\{(\overrightarrow{\textit{\textbf{n}}},m,I_{1},J_{1},I_{2},J_{2});\quad (\overrightarrow{\textit{\textbf{n}}},m,I_{1},J_{1})\in \mathbb{X}^{l_{1},r_{1}}\cap \mathbb{X}^{l_{1},r_{2}}, (\overrightarrow{\textit{\textbf{n}}},m,I_{2},J_{2})\in \mathbb{X}^{l_{2},s_{1}}\cap \mathbb{X}^{l_{2},s_{2}}  \}\cap \mathbb{D}^{t}, $$
where we are using the abbreviations $\overrightarrow{l}=(l_{1},l_{2})$, $\overrightarrow{r}=(r_{1},r_{2})$ and
$\overrightarrow{s}=(s_{1},s_{2})$. Hence we can bound the form $\widetilde{\Lambda}_{k_{1},k_{2}}$ by

\begin{equation}\label{ineq1-2nov22}
|\widetilde{\Lambda}_{k_{1},k_{2}}(f,g,H)|\lesssim\sum_{\substack{\overrightarrow{l},\overrightarrow{r},\overrightarrow{s}\in\mathbb{Z}^{2} \\ t\in\mathbb{Z}}}2^{-l_{1}}\cdot 2^{-l_{2}}\cdot 2^{-t}\cdot\#\mathbb{X}^{\overrightarrow{l},\overrightarrow{r},\overrightarrow{s},t}.
\end{equation}

The sums in \eqref{ineq1-2nov22} are not really over all integers. For instance, if $\mathbb{X}^{\overrightarrow{l},\overrightarrow{r},\overrightarrow{s},t}\neq\emptyset$, one has
$$2^{-l_{1}}\lesssim 2^{-\frac{(k_{1}+k_{2})}{2}},$$
so we can assume without loss of generality that $l_{1}\geq 0$ (and similarly for the other sums).

The following lemma plays a crucial role in the argument by relating the scalar and mixed-norm quantities involved in the problem.

\begin{lemma}If $\mathbb{X}^{\overrightarrow{l},\overrightarrow{r},\overrightarrow{s},t}\neq\emptyset$, then
\begin{equation}
    \eqalign{
\displaystyle 2^{-l_{1}}\lesssim 2^{-r_{1}}, &\displaystyle\quad 2^{-l_{1}}\lesssim \frac{2^{-r_{2}}}{\|\mathbbm{1}_{\mathbb{X}^{l_{1},r_{2}}}\|^{\frac{1}{2}}_{\ell^{\infty}_{n_{2},m,J_{1}}\ell^{1}_{n_{1},I_{1}}}}, \cr
\displaystyle 2^{-l_{2}}\lesssim \frac{2^{-s_{1}}}{\|\mathbbm{1}_{\mathbb{X}^{l_{2},s_{1}}}\|^{\frac{1}{2}}_{\ell^{\infty}_{n_{1},m,I_{2}}\ell^{1}_{n_{2},J_{2}}}}, &\displaystyle\quad 2^{-l_{2}}\lesssim 2^{-s_{2}}. \cr
}
\end{equation}
\end{lemma}
\begin{proof} By the triangle inequality and Cauchy-Schwarz,
\begin{equation*}
\eqalign{
    2^{-l_{1}}\approx\displaystyle |\langle f(\cdot)e^{-2\pi im|\cdot|^{2}},\varphi_{I_{1}\times J_{1}}^{\overrightarrow{\textit{\textbf{n}}}}\rangle| &\displaystyle\leq \frac{1}{|J_{1}|^{\frac{1}{2}}}\cdot\left\|\left\langle f^{\xi_{2}}(\cdot)e^{-2\pi im|\cdot|^{2}},\varphi^{n_{1}}_{I_{1}}\right\rangle_{\xi_{1}}\right\|_{L^{1}_{\xi_{2}}(J_{1})} \cr
    &\displaystyle\lesssim \frac{1}{|J_{1}|^{\frac{1}{2}}}\cdot \left\|\left\langle f^{\xi_{2}}(\cdot)e^{-2\pi im|\cdot|^{2}},\varphi^{n_{1}}_{I_{1}}\right\rangle_{\xi_{1}}\right\|_{L^{2}_{\xi_{2}}}\cdot |J_{1}|^{\frac{1}{2}} \cr
    &\displaystyle \approx 2^{-r_{1}}.\cr
    }
\end{equation*}

The relation above between $2^{-l_{1}}$ and $2^{-r_{2}}$ is a consequence of orthogonality: for a fixed $(n_{2},m,J_{1})$, define 
$$\mathbb{X}^{l_{1},r_{2}}_{(n_{2},m,J_{1})} :=\{ (n_{1},I_{1});\quad (\overrightarrow{\textit{\textbf{n}}},m,I_{1},J_{1})\in \mathbb{X}^{l_{1},r_{2}}\}. $$

This way,
\begin{equation*}
\eqalign{
\#\mathbb{X}^{l_{1},r_{2}}_{(n_{2},m,J_{1})} &\approx\displaystyle 2^{2l_{1}}\sum_{(n_{1},I_{1})\in \mathbb{X}^{l_{1},r_{2}}_{(n_{2},m,J_{1})}} |\langle f(\cdot)e^{-2\pi im|\cdot|^{2}},\varphi_{I_{1}\times J_{1}}^{\overrightarrow{\textit{\textbf{n}}}}\rangle|^{2}\cr
&\displaystyle\displaystyle\leq 2^{2l_{1}}\sum_{I_{1}\in\mathcal{D}^{k_{1}}_{[0,1]}}\sum_{n_{1}\in\mathbb{Z}}\left|\int_{0}^{1} \left\langle f^{\xi_{1}}(\cdot)e^{-2\pi im|\cdot|^{2}},\varphi^{n_{2}}_{J_{1}}\right\rangle_{\xi_{2}}\cdot e^{-2\pi im\xi_{1}^{2}}\cdot \varphi^{n_{1}}_{I_{1}}(\xi_{1})\cdot\mathrm{d}\xi_{1}\right|^2 \cr
&\leq\displaystyle 2^{2l_{1}}\sum_{I_{1}\in\mathcal{D}^{k_{1}}_{[0,1]}}\int_{\textnormal{supp}(\varphi_{I_{1}})} \left|\varphi_{I_{1}}(x_{1})\left\langle f^{x_{1}}(\cdot)e^{-2\pi im|\cdot|^{2}},\varphi^{n_{2}}_{J_{1}}\right\rangle_{\xi_{2}}\right|^{2}\mathrm{d}x_{1} \cr
&\lesssim\displaystyle 2^{2l_{1}}\int_{0}^{1} \left|\left\langle f^{x_{1}}(\cdot)e^{-2\pi im|\cdot|^{2}},\varphi^{n_{2}}_{J_{1}}\right\rangle_{\xi_{2}}\right|^{2}\mathrm{d}x_{1} \cr
&\approx\displaystyle 2^{2l_{1}}\cdot 2^{-2r_{2}},
}
\end{equation*}
where we used Bessel's inequality from the second to the third line. By taking the supremum in $(n_{2},m,J_{1})$ we conclude that
$$2^{-l_{1}}\lesssim \frac{ 2^{-r_{2}}}{\|\mathbbm{1}_{\mathbb{X}^{l_{1},r_{2}}}\|^{\frac{1}{2}}_{\ell^{\infty}_{n_{2},m,J_{1}}\ell^{1}_{n_{1},I_{1}}}}. $$

The other two relations are verified analogously.
\end{proof}

We will need the following convex combinations of the bounds above:

\begin{equation}\label{bound1-22oct22}
2^{-l_{1}}\lesssim \frac{2^{-\frac{1}{2}r_{1}}\cdot 2^{-\frac{1}{2}r_{2}}}{\|\mathbbm{1}_{\mathbb{X}^{l_{1},r_{2}}}\|^{\frac{1}{4}}_{\ell^{\infty}_{n_{2},m,J_{1}}\ell^{1}_{n_{1},I_{1}}}},
\end{equation}

\begin{equation}\label{bound2-22oct22}
2^{-l_{2}}\lesssim \frac{2^{-\frac{1}{2}s_{1}}\cdot 2^{-\frac{1}{2}s_{2}}}{\|\mathbbm{1}_{\mathbb{X}^{l_{2},s_{1}}}\|^{\frac{1}{4}}_{\ell^{\infty}_{n_{1},m,I_{2}}\ell^{1}_{n_{2},J_{2}}}}.
\end{equation}

We now concentrate on estimating the right-hand side of \eqref{ineq1-2nov22} by finding good bounds for $\#\mathbb{X}^{\overrightarrow{l},\overrightarrow{r},\overrightarrow{s},t}$. Observe that the Fourier transform of
\begin{equation}\label{bumps-5nov22}
\varphi^{I_{1}\times J_{1}}_{\overrightarrow{\textit{\textbf{n}}}}\cdot \varphi^{I_{2}\times J_{2}}_{\overrightarrow{\textit{\textbf{n}}}}
\end{equation}
is essentially supported on $I_{1}\times J_{1}+I_{2}\times J_{2}$. Since the scales of $I_{1}, I_{2}, J_{1}, J_{2}$ are fixed, these supports have finite overlap. Notice also that the bumps in \eqref{bumps-5nov22} are $L^{1}$-normalized, hence we conclude that the set
$$\left\{2^{\frac{(k_{1}+k_{2})}{2}}\cdot\left(\varphi^{I_{1}\times J_{1}}_{\overrightarrow{\textit{\textbf{n}}}}\cdot \varphi^{I_{2}\times J_{2}}_{\overrightarrow{\textit{\textbf{n}}}}\right)\right\}_{\overrightarrow{\textit{\textbf{n}}},I_{1},J_{1},I_{2},J_{2}}$$
is an $L^{2}$-normalized almost-orthogonal family\footnote{Recall that we are treating the term with $\#=0$ from \eqref{ineq2-5nov22}. If $\#>0$, the family
$$\left\{\#^{50}\cdot 2^{\frac{(k_{1}+k_{2})}{2}}\cdot\left(\varphi^{I_{1}\times J_{1}}_{\overrightarrow{\textit{\textbf{n}}}}\cdot \varphi^{I_{2}\times J_{2}}_{\overrightarrow{\textit{\textbf{n}}}}\right)\right\}_{\overrightarrow{\textit{\textbf{n}}},I_{1},J_{1},I_{2},J_{2}}$$
would be $L^{2}$-normalized, which would imply \eqref{ineq1-oct1822} with an extra factor of $\#^{-100}$. The geometric series in $\#$ is then summable.}. This way,

\begin{equation}\label{ineq1-oct1822}
\#\mathbb{X}^{\overrightarrow{l},\overrightarrow{r},\overrightarrow{s},t}\lesssim 2^{2t}\cdot 2^{-(k_{1}+k_{2})}\sum_{\substack{(\overrightarrow{\textit{\textbf{n}}},m,I_{1},J_{1},I_{2},J_{2}) \\ \in\textnormal{ }\mathbb{X}^{\overrightarrow{l},\overrightarrow{r},\overrightarrow{s},t}}}\frac{1}{2^{-(k_{1}+k_{2})}}|\langle H,\varphi^{I_{1}\times J_{1}}_{\overrightarrow{\textit{\textbf{n}}}}\cdot \varphi^{I_{2}\times J_{2}}_{\overrightarrow{\textit{\textbf{n}}}}\otimes \chi_{m}\rangle|^{2}\leq 2^{2t}\cdot 2^{-(k_{1}+k_{2})}\cdot |F|.
\end{equation}

Alternatively,
\begin{equation}\label{bound1-6nov22}
\eqalign{
    \#\mathbb{X}^{\overrightarrow{l},\overrightarrow{r},\overrightarrow{s},t}&\displaystyle\lesssim 2^{t}\cdot\sum_{\substack{(\overrightarrow{\textit{\textbf{n}}},m,I_{1},J_{1},I_{2},J_{2}) \\ \in\textnormal{ }\mathbb{X}^{\overrightarrow{l},\overrightarrow{r},\overrightarrow{s},t}}}|\langle H,\varphi^{I_{1}\times J_{1}}_{\overrightarrow{\textit{\textbf{n}}}}\cdot \varphi^{I_{2}\times J_{2}}_{\overrightarrow{\textit{\textbf{n}}}}\otimes \chi_{m}\rangle| \cr
    &\displaystyle\leq 2^{t}\cdot\sum_{\substack{(I_{1},J_{1},I_{2},J_{2}) \\ I_{1}\in\mathcal{D}_{[0,1]}^{k_{1}} \\ J_{1}\in\mathcal{D}_{[0,1]}^{k_{2}}}} \sum_{(\overrightarrow{\textit{\textbf{n}}},m)\in\mathbb{Z}^{3}}|\langle H,\varphi^{I_{1}\times J_{1}}_{\overrightarrow{\textit{\textbf{n}}}}\cdot \varphi^{I_{2}\times J_{2}}_{\overrightarrow{\textit{\textbf{n}}}}\otimes \chi_{m}\rangle| \cr
    &\displaystyle\lesssim 2^{t}\cdot\sum_{\substack{(I_{1},J_{1},I_{2},J_{2}) \\ I_{1}\in\mathcal{D}_{[0,1]}^{k_{1}} \\ J_{1}\in\mathcal{D}_{[0,1]}^{k_{2}}}} \frac{1}{2^{k_{1}+k_{2}}}\|H\|_{1} \cr
    &\displaystyle\lesssim 2^{t}\cdot |F|,\cr
    }
\end{equation}
where we took into account the fact that the family \eqref{bumps-5nov22} is $L^{1}$-normalized and that there are $2^{k_{1}+k_{2}}$ elements in the sum over $(I_{1},J_{1},I_{2},J_{2})$\footnote{We remark that there is no condition on $I_{2}$ and $J_{2}$ in the sum over $(I_{1},J_{1},I_{2},J_{2})$ because, again, they are determined by $I_{1}$ and $J_{1}$, respectively.}.

From the definition of $\mathbb{X}^{\overrightarrow{l},\overrightarrow{r},\overrightarrow{s},t}$ it follows that:
\begin{equation}\label{bound3-22oct22}
\#\mathbb{X}^{\overrightarrow{l},\overrightarrow{r},\overrightarrow{s},t} \leq \|\mathbbm{1}_{\mathbb{X}^{l_{2},s_{1}}}\|_{\ell^{\infty}_{n_{1},m,I_{2}}\ell^{1}_{n_{2},J_{2}}} \cdot \|\mathbbm{1}_{\mathbb{B}^{r_{1}}_{1}\cap\mathbb{C}^{s_{1}}_{1}}\|_{\ell^{1}_{n_{1},m,I_{1}}},
\end{equation}

\begin{equation}\label{bound4-22oct22}
\#\mathbb{X}^{\overrightarrow{l},\overrightarrow{r},\overrightarrow{s},t} \leq \|\mathbbm{1}_{\mathbb{X}^{l_{1},r_{2}}}\|_{\ell^{\infty}_{n_{2},m,J_{1}}\ell^{1}_{n_{1},I_{1}}}\cdot \|\mathbbm{1}_{\mathbb{B}^{r_{2}}_{2}\cap\mathbb{C}^{s_{2}}_{2} }\|_{\ell^{1}_{n_{2},m,J_{2}}}.
\end{equation}

The next step will make it clear what is the gain obtained from our previous considerations: by squaring the extension operator $\mathcal{E}_{2}$ and performing appropriate Whitney decompositions, one obtains objects analogous to the bilinear extension operators discussed in \cite{MO}. For such operators, one can take advantage of \textit{transversality} to prove better bounds than the ones satisfied by $\mathcal{E}_{2}$, which will be done here through Proposition \ref{bilinear1d-6nov22}. Observe that

\begin{equation}\label{bound5-22oct22}
\eqalign{
\#\mathbb{B}^{r_{1}}_{1}\cap\mathbb{C}^{s_{1}}_{1}&\displaystyle\lesssim 2^{2r_{1}+2s_{1}}\sum_{I_{1}\in\mathcal{D}_{[0,1]}^{k_{1}}}\sum_{n_{1},m\in\mathbb{Z}} \left\|\left\langle f^{\xi_{2}}(\cdot)e^{-2\pi im|\cdot|^{2}},\varphi^{n_{1}}_{I_{1}}\right\rangle_{\xi_{1}}\right\|_{L^{2}_{\xi_{2}}}^{2}\cdot \left\|\left\langle g^{\xi_{2}}(\cdot)e^{-2\pi im|\cdot|^{2}},\varphi^{n_{1}}_{I_{2}}\right\rangle_{\xi_{1}}\right\|_{L^{2}_{\xi_{2}}}^{2}\cr
&=2^{2r_{1}+2s_{1}}\displaystyle\sum_{I_{1}\in\mathcal{D}_{[0,1]}^{k_{1}}}\displaystyle\int_{[0,1]^{2}}\left(\sum_{n_{1},m\in\mathbb{Z}}\left|\left\langle f^{\xi_{2}}(\cdot)e^{-2\pi im|\cdot|^{2}},\varphi^{n_{1}}_{I_{1}}\right\rangle_{\xi_{1}}\right|^{2}\cdot\left|\left\langle g^{\widetilde{\xi_{2}}}(\cdot)e^{-2\pi im|\cdot|^{2}},\varphi^{n_{1}}_{I_{2}}\right\rangle_{\xi_{1}}\right|^{2}\right)\mathrm{d}\xi_{2}\mathrm{d}\widetilde{\xi_{2}} \cr
&=2^{2r_{1}+2s_{1}}\displaystyle\sum_{I_{1}\in\mathcal{D}_{[0,1]}^{k_{1}}}\int_{[0,1]^{2}} \frac{1}{2^{-2k_{1}}}\cdot\|f^{\xi_{2}}\|_{L^{2}(I_{1})}^{2}\cdot \|g^{\widetilde{\xi_{2}}}\|_{L^{2}(I_{2})}^{2}\mathrm{d}\xi_{2}\mathrm{d}\widetilde{\xi_{2}} \cr
&\leq\displaystyle 2^{2r_{1}+2s_{1}}\cdot 2^{2k_{1}}\sum_{I_{1}\in\mathcal{D}_{[0,1]}^{k_{1}}}\int_{I_{1}\times [0,1]\times I_{2}\times [0,1]}\mathbbm{1}_{E_{1}}\otimes\mathbbm{1}_{E_{2}} \cr
&\leq \displaystyle 2^{2r_{1}+2s_{1}}\cdot 2^{2k_{1}}\cdot |E_{1}|\cdot |E_{2}|,
}
\end{equation}\label{bound6-22oct22}
by Proposition \ref{bilinear1d-6nov22}. Analogously,
\begin{equation}
\#\mathbb{B}^{r_{2}}_{2}\cap\mathbb{C}^{s_{2}}_{2}\displaystyle\lesssim 2^{2r_{2}+2s_{2}}\cdot 2^{2k_{2}}\cdot |E_{1}|\cdot |E_{2}|.
\end{equation}

Back to the multilinear form, using \eqref{bound1-22oct22}, \eqref{bound2-22oct22}, interpolating between \eqref{bound3-22oct22}, \eqref{bound4-22oct22}, \eqref{ineq1-oct1822} and \eqref{bound1-6nov22},

\begin{equation*}
\eqalign{
|\widetilde{\Lambda}_{k_{1},k_{2}}&(f,g,H)| \cr &\lesssim\displaystyle\sum_{\substack{\overrightarrow{l},\overrightarrow{r} \\ \overrightarrow{s},t\geq 0}}\quad\displaystyle 2^{-\varepsilon l_{1}}\cdot\left(\frac{2^{-\frac{1}{2}r_{1}}\cdot 2^{-\frac{1}{2}r_{2}}}{\|\mathbbm{1}_{\mathbb{X}^{l_{1},r_{2}}}\|^{\frac{1}{4}}_{\ell^{\infty}_{n_{2},m,J_{1}}\ell^{1}_{n_{1},I_{1}}}}\right)^{1-\varepsilon}  \cdot 2^{-\varepsilon l_{2}}\cdot\left(\frac{2^{-\frac{1}{2}s_{1}}\cdot 2^{-\frac{1}{2}s_{2}}}{\|\mathbbm{1}_{\mathbb{X}^{l_{2},s_{1}}}\|^{\frac{1}{4}}_{\ell^{\infty}_{n_{1},m,I_{2}}\ell^{1}_{n_{2},J_{2}}}}\right)^{1-\varepsilon} \cdot 2^{-t} \cr
&\displaystyle\qquad\qquad\cdot \left(\|\mathbbm{1}_{\mathbb{X}^{l_{2},s_{1}}}\|_{\ell^{\infty}_{n_{1},m,I_{2}}\ell^{1}_{n_{2},J_{2}}} \cdot \|\mathbbm{1}_{\mathbb{B}^{r_{1}}_{1}\cap\mathbb{C}^{s_{1}}_{1}}\|_{\ell^{1}_{n_{1},m,I_{1}}}\right)^{\theta_{1}} \cr
&\displaystyle\qquad\qquad\cdot\left(\|\mathbbm{1}_{\mathbb{X}^{l_{1},r_{2}}}\|_{\ell^{\infty}_{n_{2},m,J_{1}}\ell^{1}_{n_{1},I_{1}}}\cdot \|\mathbbm{1}_{\mathbb{B}^{r_{2}}_{2}\cap\mathbb{C}^{s_{2}}_{2} }\|_{\ell^{1}_{n_{2},m,J_{2}}}\right)^{\theta_{2}} \cr
&\displaystyle\qquad\qquad\cdot\left(2^{2t}\cdot 2^{-(k_{1}+k_{2})}\cdot |F|\right)^{\theta_{3}} \cr
&\displaystyle\qquad\qquad\cdot\left(2^{t}\cdot |F|\right)^{\theta_{4}}. \cr
}
\end{equation*}

Choosing $\theta_{4}=3\varepsilon$,  $\theta_{3}=\frac{1}{2}-2\varepsilon$, $\theta_{2}=\frac{1}{4}-\frac{\varepsilon}{2}$, $\theta_{1}=\frac{1}{4}-\frac{\varepsilon}{2}$ and plugging in \eqref{bound5-22oct22} and \eqref{bound6-22oct22},
\begin{equation*}
\eqalign{
|\widetilde{\Lambda}_{k_{1},k_{2}}&(f,g,H)| \cr &\displaystyle\lesssim\displaystyle\sum_{\substack{\overrightarrow{l},\overrightarrow{r} \\ \overrightarrow{s},t\geq 0}}\quad\displaystyle 2^{-\varepsilon l_{1}}\cdot\left(\frac{2^{-\frac{1}{2}r_{1}}\cdot 2^{-\frac{1}{2}r_{2}}}{\|\mathbbm{1}_{\mathbb{X}^{l_{1},r_{2}}}\|^{\frac{1}{4}}_{\ell^{\infty}_{n_{2},m,J_{1}}\ell^{1}_{n_{1},I_{1}}}}\right)^{1-\varepsilon}  \cdot 2^{-\varepsilon l_{2}}\cdot\left(\frac{2^{-\frac{1}{2}s_{1}}\cdot 2^{-\frac{1}{2}s_{2}}}{\|\mathbbm{1}_{\mathbb{X}^{l_{2},s_{1}}}\|^{\frac{1}{4}}_{\ell^{\infty}_{n_{1},m,I_{2}}\ell^{1}_{n_{2},J_{2}}}}\right)^{1-\varepsilon} \cdot 2^{-t} \cr
&\qquad\qquad\qquad\displaystyle\cdot \left(\|\mathbbm{1}_{\mathbb{X}^{l_{2},s_{1}}}\|_{\ell^{\infty}_{n_{1},m,I_{2}}\ell^{1}_{n_{2},J_{2}}} \cdot 2^{2r_{1}+2s_{1}}\cdot 2^{2k_{1}}\cdot |E_{1}|\cdot |E_{2}|\right)^{\frac{1}{4}-\frac{\varepsilon}{2}} \cr
&\qquad\qquad\qquad\displaystyle\cdot\left(\|\mathbbm{1}_{\mathbb{X}^{l_{1},r_{2}}}\|_{\ell^{\infty}_{n_{2},m,J_{1}}\ell^{1}_{n_{1},I_{1}}}\cdot 2^{2r_{2}+2s_{2}}\cdot 2^{2k_{2}}\cdot |E_{1}|\cdot |E_{2}|\right)^{\frac{1}{4}-\frac{\varepsilon}{2}} \cr
&\qquad\qquad\qquad\displaystyle\cdot\left(2^{2t}\cdot 2^{-(k_{1}+k_{2})}|F|\right)^{\frac{1}{2}-2\varepsilon}
\cr
&\qquad\qquad\qquad\displaystyle\cdot\left(2^{t}\cdot |F|\right)^{3\varepsilon} \cr
&\lesssim\displaystyle\sum_{\substack{\overrightarrow{l},\overrightarrow{r} \\ \overrightarrow{s},t\geq 0}} 2^{-\varepsilon l_{1}}\cdot 2^{-\varepsilon l_{2}}\cdot 2^{-\frac{\varepsilon}{2} r_{1}}\cdot 2^{-\frac{\varepsilon}{2} r_{2}}\cdot 2^{-\frac{\varepsilon}{2} s_{1}}\cdot 2^{-\frac{\varepsilon}{2} s_{2}}\cdot 2^{-\varepsilon t}\cdot 2^{\varepsilon\left(k_{1}+k_{2}\right)}\cr
&\qquad\qquad\qquad\displaystyle\cdot \left(\|\mathbbm{1}_{\mathbb{X}^{l_{1},r_{2}}}\|^{\frac{1}{4}}_{\ell^{\infty}_{n_{2},m,J_{1}}\ell^{1}_{n_{1},I_{1}}}\cdot \|\mathbbm{1}_{\mathbb{X}^{l_{2},s_{1}}}\|^{\frac{1}{4}}_{\ell^{\infty}_{n_{1},m,I_{2}}\ell^{1}_{n_{2},J_{2}}}\right)^{-\frac{\varepsilon}{4}} \cr
&\qquad\qquad\qquad\displaystyle\cdot |E_{1}|^{\frac{1}{2}-\varepsilon}\cdot |E_{2}|^{\frac{1}{2}-\varepsilon}\cdot |F|^{\frac{1}{2}+\varepsilon}. \cr
}
\end{equation*}

Recall that the sums run over indices $l_{1},l_{2},r_{1},r_{2},s_{1},s_{2}$ such that

$$2^{-l_{1}}\lesssim 2^{-\frac{(k_{1}+k_{2})}{2}},$$
$$2^{-l_{2}} \lesssim 2^{-\frac{(k_{1}+k_{2})}{2}}.$$
by the triangle inequality and
$$2^{-r_{1}}\lesssim 2^{-\frac{k_{1}}{2}},\quad 2^{-r_{2}}\lesssim 2^{-\frac{k_{2}}{2}},\quad 2^{-s_{1}}\lesssim 2^{-\frac{k_{1}}{2}},\quad 2^{-s_{2}}\lesssim 2^{-\frac{k_{2}}{2}}. $$
since
\begin{equation*}
\eqalign{
\displaystyle 2^{-2r_{1}}\lesssim \int_{0}^{1}\left|\int_{0}^{1}\frac{1}{|I_{1}|^{\frac{1}{2}}}\cdot |f(\xi_{1},\xi_{2})|\cdot\varphi_{I_{1}}(\xi_{1})\mathrm{d}\xi_{1}\right|^{2}\mathrm{d}\xi_{2}&\displaystyle\lesssim \int_{0}^{1}\int_{\textnormal{supp}(\varphi_{I_{1}})} |\mathbbm{1}_{E_{1}}(\xi_{1},\xi_{2})|\mathrm{d}\xi_{1}\mathrm{d}\xi_{2} \cr
&\displaystyle\lesssim |E_{1}\cap (I_{1}\times [0,1])| \cr
&\displaystyle\leq 2^{-k_{1}}, \cr
}
\end{equation*}
and analogously for $r_{2},s_{1},s_{2}$. Therefore the sum above is bounded by

\begin{equation}\label{almostfinal2-6nov22}
    |\widetilde{\Lambda}_{k_{1},k_{2}}(f,g,H)| \lesssim_{\varepsilon} 2^{-\frac{\varepsilon}{2}\left(k_{1}+k_{2}\right)}\cdot |E_{1}|^{\frac{1}{2}-\frac{\varepsilon}{2}}\cdot |E_{2}|^{\frac{1}{2}-\frac{\varepsilon}{2}}\cdot |F|^{\frac{1}{2}+\varepsilon}.
\end{equation}

We also have the trivial estimate

\begin{equation}\label{almostfinal3-18feb23}
    |\widetilde{\Lambda}_{k_{1},k_{2}}(f,g,H)| \lesssim |E_{1}|\cdot |E_{2}|\cdot |F|.
\end{equation}

Interpolating between \eqref{almostfinal2-6nov22} and \eqref{almostfinal3-18feb23} with weights $\gamma_{1}=\frac{1}{1+2\varepsilon}$ and $\gamma_{2}=\frac{2\varepsilon}{1+2\varepsilon}$, respectively, gives

\begin{equation}\label{almostfinal4-18feb23}
\eqalign{
    |\widetilde{\Lambda}_{k_{1},k_{2}}(f,g,H)| &\displaystyle\lesssim_{\varepsilon} 2^{-\frac{\varepsilon}{2(1+2\varepsilon)}\left(k_{1}+k_{2}\right)}\cdot |E_{1}|^{\frac{1}{2}}\cdot |E_{2}|^{\frac{1}{2}}\cdot |F|^{\frac{1}{2}+\frac{2\varepsilon}{1+2\varepsilon}} \cr
    &\displaystyle \lesssim_{\varepsilon} 2^{-\frac{\varepsilon}{3}\left(k_{1}+k_{2}\right)}\cdot |E_{1}|^{\frac{1}{2}}\cdot |E_{2}|^{\frac{1}{2}}\cdot |F|^{\frac{1}{2}+\frac{2\varepsilon}{1+2\varepsilon}}. \cr
    }
\end{equation}

Restricted weak-type interpolation shows that $\widetilde{T}_{k_{1},k_{2}}$ maps $L^{2}\times L^{2}$ to $L^{2+\frac{8\varepsilon}{1-2\varepsilon}}$ with operatorial bound $O(2^{-\frac{\varepsilon}{3}(k_{1}+k_{2})})$, as desired.
\end{proof}


\Addresses

\end{document}